\documentclass[oneside,english]{amsart}
\usepackage[T1]{fontenc}
\usepackage[latin9]{inputenc}
\usepackage{amsthm}
\usepackage{amssymb}
\usepackage{esint}

\makeatletter
\numberwithin{equation}{section}
\numberwithin{figure}{section}
\theoremstyle{plain}
\newtheorem{thm}{Theorem}
  \theoremstyle{plain}
  \newtheorem{lem}[thm]{Lemma}

\makeatother

\usepackage{babel}

\begin{document}

\title{On Asymptotics Of $\Gamma_{q}(z)$ As $q$ Approaching $1$}

\author{Ruiming Zhang}
\begin{abstract}
In this note we give a derivation of the asymptotic formula for the
$q$-Gamma function as $q$ approaching $1$. This formula is valid
on all the complex plan except at the poles of the Euler Gamma function.
\end{abstract}

\subjclass[2000]{Primary 33D05; Secondary 33D15. }

\curraddr{College of Science\\
Northwest A\&F University\\
Yangling, Shaanxi 712100\\
P. R. China.}

\keywords{\noindent Gamma Function $\Gamma(z)$; $q$-Gamma Function $\Gamma_{q}(z)$;
Asymptotics.}

\email{ruimingzhang@yahoo.com}

\maketitle

\section{Introduction}

Recall that the $q$-Gamma function is defined as \cite{Andrews,Gasper,Ismail}\begin{align*}
\Gamma_{q}(z) & =\frac{(q;q)_{\infty}}{(1-q)^{z-1}(q^{z};q)_{\infty}},\end{align*}
where\begin{align*}
\left(a;q\right)_{\infty} & =\prod_{k=0}^{\infty}\left(1-aq^{k}\right),\quad a\in\mathbb{C},\ q\in\left(0,1\right).\end{align*}
$\Gamma_{q}(z)$ is fundamental to the theory of basic hypergeometric
series. It is known that \cite{Andrews,Gasper,Ismail} \begin{align*}
\lim_{q\to1}\Gamma_{q}(z) & =\Gamma(z),\end{align*}
where $\Gamma(z)$ is the Euler Gamma function. This fact is the link
between the basic hypergeometric series and the classical hypergeometric
series. All the standard textbooks on $q$-series presented W. Gosper's
heuristic argument. A rigorous version of Gosper's argument and an
alternative proof were later found by T. Koornwinder, \cite{Andrews,Koornwinder1},
but these proofs failed to indicate the speed of convergence. In \cite{Zhang}
we presented a proof by using a $q$-Beta integral. In this note we
will give yet another asymptotic formula valid on the entire complex
plane except at poles of $\Gamma(z)$. Our proof only uses calculus
for the $\Re(z)>0$ , then apply the transformation formula of $\theta_{1}(z;q)$
to get the case for $\Re(z)<1$.

\section{Main Results}

For the sake of completeness we give a proof for the following Lemma,
also see  \cite{Kirillov,Koornwinder2}.
\begin{lem}
Let\begin{align*}
\left|z\right| & <1,\ 0<q<1,\end{align*}
then\begin{align}
\left(z;q\right)_{\infty} & =\exp\left\{ -\sum_{k=1}^{\infty}\frac{1}{1-q^{k}}\frac{z^{k}}{k}\right\} .\label{eq:2.1}\end{align}
\end{lem}
\begin{proof}
From\begin{align*}
\log(1-z) & =-\sum_{k=1}^{\infty}\frac{z^{k}}{k},\quad\left|z\right|<1\end{align*}
to get

\begin{eqnarray*}
\log\left(z,q\right)_{\infty} & = & \sum_{j=0}^{\infty}\log\left(1-zq^{j}\right)=-\sum_{j=0}^{\infty}\sum_{k=1}^{\infty}\frac{\left(zq^{j}\right)^{k}}{k}\\
 & = & -\sum_{k=1}^{\infty}\frac{z^{k}}{k}\sum_{j=0}^{\infty}q^{jk}=-\sum_{k=1}^{\infty}\frac{z^{k}}{k\left(1-q^{k}\right)}\end{eqnarray*}
for $q\in(0,1)$, where all the logarithms are taken as their principle
branches. \eqref{eq:2.1} is obtained by taking exponentials. \end{proof}
\begin{lem}
Let \begin{align*}
q & =e^{-\pi\tau},\ \tau>0,\quad\Re(w)>0,\end{align*}
then,\begin{align}
\left(q^{w+1};q\right)_{\infty} & =\frac{\sqrt{2\pi}w^{w-1/2}\exp\left(-\frac{\pi}{6\tau}\right)}{\Gamma\left(w\right)\left(1-e^{-\tau\pi w}\right)^{w+1/2}}\left\{ 1+\mathcal{O}\left(\tau\right)\right\} ,\label{eq:2.2}\end{align}
as $\tau\to0^{+}$.\end{lem}
\begin{proof}
Take $z=qe^{-\tau\pi w}$ in \eqref{eq:2.1} with $\Re(w)>0$ to obtain\[
\left(qe^{-\tau\pi w},q\right)_{\infty}=\exp\left\{ -\sum_{k=1}^{\infty}\frac{q^{k}e^{-k\tau\pi w}}{k\left(1-q^{k}\right)}\right\} \]
and \begin{align*}
\sum_{k=1}^{\infty}\frac{q^{k}e^{-k\tau\pi w}}{k\left(1-q^{k}\right)} & =\sum_{k=1}^{\infty}\frac{e^{-k\tau\pi w}}{k}\left\{ \frac{q^{k}}{1-q^{k}}-\frac{1}{k\pi\tau}+\frac{1}{2}-\frac{k\pi\tau}{12}\right\} \\
 & +\sum_{k=1}^{\infty}\frac{e^{-k\tau\pi w}}{k}\left\{ \frac{1}{k\pi\tau}-\frac{1}{2}+\frac{k\pi\tau}{12}\right\} \\
 & =S+\frac{1}{\pi\tau}\sum_{k=1}^{\infty}\frac{e^{-k\tau\pi w}}{k^{2}}-\frac{1}{2}\sum_{k=1}^{\infty}\frac{e^{-k\tau\pi w}}{k}+\frac{\pi\tau}{12}\sum_{k=1}^{\infty}e^{-k\tau\pi w}\\
 & =S+\frac{1}{\pi\tau}\mbox{Li}_{2}\left(\exp(-\pi\tau w)\right)+\frac{1}{2}\log\left(1-e^{-\tau\pi w}\right)+\frac{\pi\tau}{12\left(\exp(\tau\pi w)-1\right)},\end{align*}
where\[
S=\sum_{k=1}^{\infty}\frac{e^{-k\tau\pi w}}{k}\left\{ \frac{1}{e^{k\pi\tau}-1}-\frac{1}{k\pi\tau}+\frac{1}{2}-\frac{k\pi\tau}{12}\right\} \]
and \[
\mbox{Li}_{2}\left(z\right)=\sum_{k=1}^{\infty}\frac{z^{n}}{n^{2}},\quad\left|z\right|\le1.\]
From \cite{Andrews} \begin{align*}
\mbox{Li}_{2}\left(z\right) & =-\mbox{Li}_{2}\left(1-z\right)+\frac{\pi^{2}}{6}-\log z\cdot\log\left(1-z\right)\end{align*}
to get\begin{align*}
\mbox{Li}_{2}\left(\exp(-\pi\tau w)\right) & =-\mbox{Li}_{2}\left(1-\exp(-\pi\tau w)\right)+\frac{\pi^{2}}{6}+\pi\tau w\log\left(1-\exp(-\pi\tau w)\right)\\
 & =-\pi\tau w+\frac{\pi^{2}}{6}+\pi\tau w\log\left(1-\exp(-\pi\tau w)\right)+\mathcal{O}\left(\tau^{2}\right),\end{align*}
hence\begin{align*}
\sum_{k=1}^{\infty}\frac{q^{k}e^{-k\tau\pi w}}{k\left(1-q^{k}\right)} & =S-w+\frac{\pi}{6\tau}+\left(w+\frac{1}{2}\right)\log\left(1-\exp(-\pi\tau w)\right)+\frac{\pi\tau}{12\left(\exp(\tau\pi w)-1\right)}+\mathcal{O}\left(\tau\right)\end{align*}
as $\tau\to0^{+}$. From \cite{Andrews}\begin{align*}
\log\Gamma(w) & =\left(w-\frac{1}{2}\right)\log w-w+\frac{\log(2\pi)}{2}+\int_{0}^{\infty}\left(\frac{1}{2}-\frac{1}{t}+\frac{1}{e^{t}-1}\right)\frac{e^{-tw}}{t}dt\end{align*}
to obtain\begin{align*}
 & \int_{0}^{\infty}\left(\frac{1}{2}-\frac{1}{t}-\frac{t}{12}+\frac{1}{e^{t}-1}\right)\frac{e^{-tw}}{t}dt\\
 & =\log\Gamma(w)-\left(w-\frac{1}{2}\right)\log w+w-\frac{\log(2\pi)}{2}-\frac{1}{12w}\end{align*}
for $\Re(w)>0$. Write\begin{align*}
I & =\int_{0}^{\infty}\left(\frac{1}{2}-\frac{1}{t}-\frac{t}{12}+\frac{1}{e^{t}-1}\right)\frac{e^{-tw}}{t}dt\\
 & =\sum_{k=1}^{\infty}\int_{(k-1)\pi\tau}^{k\pi\tau}\left(\frac{1}{2}-\frac{1}{t}-\frac{t}{12}+\frac{1}{e^{t}-1}\right)\frac{e^{-tw}}{t}dt\end{align*}
and\[
f(t)=\left(\frac{1}{2}-\frac{1}{t}-\frac{t}{12}+\frac{1}{e^{t}-1}\right)\frac{e^{-tw}}{t},\]
then \[
f'(t)=\mathcal{O}\left(t\right)\]
for $t\to0^{+}$ and \[
f'(t)=\mathcal{O}\left(\exp\left(-t\Re(w)\right)\right)\]
for $t\to+\infty$. Hence, \begin{align*}
S-I & =\sum_{k=1}^{\infty}\int_{(k-1)\pi\tau}^{k\pi\tau}dt\int_{t}^{k\pi\tau}f'(y)dy\\
 & =\sum_{k=1}^{\infty}\int_{(k-1)\pi\tau}^{k\pi\tau}f'(y)\int_{(k-1)\pi\tau}^{y}dtdy\\
 & =\sum_{k=1}^{\infty}\int_{(k-1)\pi\tau}^{k\pi\tau}f'(y)\left(y-(k-1)\pi\tau\right)dy,\end{align*}
thus,\[
\left|S-I\right|\le\pi\tau\int_{0}^{\infty}\left|f'(y)\right|dy\]
and\[
S-I=\mathcal{O}\left(\pi\tau\right)\]
as $\tau\to0^{+}$. Then,\begin{align*}
\sum_{k=1}^{\infty}\frac{q^{k}e^{-k\tau\pi w}}{k\left(1-q^{k}\right)} & =\log\Gamma(w)-\left(w-\frac{1}{2}\right)\log w-\frac{\log(2\pi)}{2}\\
 & +\frac{\pi\tau}{12}\left(\frac{1}{\exp(\tau\pi w)-1}-\frac{1}{\pi\tau w}\right)+\frac{\pi}{6\tau}\\
 & +\left(w+\frac{1}{2}\right)\log\left(1-\exp(-\pi\tau w)\right)+\mathcal{O}\left(\tau\right)\\
 & =\log\Gamma(w)-\left(w-\frac{1}{2}\right)\log w-\frac{\log(2\pi)}{2}\\
 & +\frac{\pi}{6\tau}+\left(w+\frac{1}{2}\right)\log\left(1-\exp(-\pi\tau w)\right)+\mathcal{O}\left(\tau\right)\end{align*}
as $\tau\to0^{+}$. \end{proof}
\begin{thm}
Let $q=\exp\left(-\pi\tau\right)$ with $\tau>0$. Then \begin{align}
\Gamma_{q}\left(w\right) & =\Gamma\left(w\right)\left\{ 1+\mathcal{O}\left(\tau\right)\right\} \label{eq:2.3}\end{align}
 as $\tau\to0^{+}$for $z\notin\mathbb{N}\cup\left\{ 0\right\} $.\end{thm}
\begin{proof}
For $\Re(w)>0$, from \eqref{eq:2.2} to get \begin{align*}
\left(q^{w};q\right)_{\infty} & =\left(1-e^{-\tau\pi w}\right)\left(qe^{-\tau\pi w},q\right)_{\infty}=\frac{\sqrt{2\pi}w^{w-1/2}\exp\left(-\frac{\pi}{6\tau}\right)}{\Gamma\left(w\right)\left(1-e^{-\tau\pi w}\right)^{w-1/2}}\left\{ 1+\mathcal{O}\left(\tau\right)\right\} \end{align*}
and \begin{align*}
\left(q;q\right)_{\infty} & =\frac{\sqrt{2\pi}\exp\left(-\frac{\pi}{6\tau}\right)}{\left(1-e^{-\tau\pi}\right)^{1/2}}\left\{ 1+\mathcal{O}\left(\tau\right)\right\} \end{align*}
as $\tau\to0^{+}$. Hence, for $\Re(w)>0$  we have

\begin{align}
\Gamma_{q}\left(w\right) & =\frac{\left(q;q\right)_{\infty}}{\left(1-q\right)^{w-1}\left(q^{w};q\right)_{\infty}}=\Gamma\left(w\right)\left\{ \frac{1-e^{-\pi\tau w}}{w\left(1-e^{-\pi\tau}\right)}\right\} ^{w-\frac{1}{2}}\left\{ 1+\mathcal{O}\left(\tau\right)\right\} \label{eq:2.4}\end{align}
as $\tau\to0^{+}$. We get \eqref{eq:2.3} from \eqref{eq:2.4} and
\[
\left\{ \frac{1-e^{-\pi\tau w}}{w\left(1-e^{-\pi\tau}\right)}\right\} ^{w-\frac{1}{2}}=1+\mathcal{O}\left(\tau\right)\]
for $\Re(w)>0$ as $\tau\to0^{+}$. Recall that \cite{Whittaker}
\[
\theta_{1}(v|t)=2\sum_{k=0}^{\infty}(-1)^{k}p^{(k+1/2)^{2}}\sin(2k+1)\pi v,\]
\[
\theta_{1}(v|t)=2p^{1/4}\sin\pi v(p^{2};p^{2})_{\infty}(p^{2}e^{2\pi iv};p^{2})_{\infty}(p^{2}e^{-2\pi iv};p^{2})_{\infty}\]
and\[
\theta_{1}\left(\frac{v}{t}\mid-\frac{1}{t}\right)=-i\sqrt{\frac{t}{i}}e^{\pi iv^{2}/t}\theta_{1}\left(v\mid t\right),\]
where $p=e^{\pi it},\quad\Im(t)>0$, then,\begin{align*}
\left(q,q^{1+x},q^{1-x};q\right)_{\infty} & =\frac{\exp\left(\frac{\pi\tau}{8}+\frac{\pi\tau x^{2}}{2}\right)\theta_{1}\left(x\vert\frac{2i}{\tau}\right)}{\sqrt{2\tau}\sinh\frac{\pi\tau x}{2}}\end{align*}
 and\[
\left(q;q\right)_{\infty}^{3}=\frac{\sqrt{2}\exp\left(\frac{\pi\tau}{8}\right)\theta'_{1}\left(0\vert\frac{2i}{\tau}\right)}{\pi\tau^{3/2}}\]
for $q=\exp(-\pi\tau)$ and $\tau>0$. Hence for $x\notin\mathbb{Z}$
we have\begin{align*}
\Gamma_{q}\left(1+x\right)\Gamma_{q}\left(1-x\right) & =\frac{\left(q;q\right)_{\infty}^{3}}{\left(q,q^{1+x},q^{1-x};q\right)_{\infty}}=\frac{2\sinh\frac{\pi\tau x}{2}\theta'_{1}\left(0\vert\frac{2i}{\tau}\right)}{\pi\tau\exp\left(\frac{\pi\tau x^{2}}{2}\right)\theta_{1}\left(x\vert\frac{2i}{\tau}\right)}\end{align*}
and\begin{align*}
\Gamma_{q}\left(x\right)\Gamma_{q}(1-x) & =\frac{1-q}{1-q^{x}}\Gamma_{q}(1+x)=\frac{\left(e^{\pi\tau}-1\right)\theta'_{1}\left(0\vert\frac{2i}{\tau}\right)}{\pi\tau\exp\left(\frac{\pi\tau(x^{2}+x+2)}{2}\right)\theta_{1}\left(x\vert\frac{2i}{\tau}\right)}.\end{align*}
From \[
\frac{e^{\pi\tau}-1}{\pi\tau}=1+\mathcal{O}\left(\tau\right),\]

\[
\theta'_{1}\left(0\vert\frac{2i}{\tau}\right)=2\pi\exp\left(-\frac{\pi}{2\tau}\right)\left\{ 1+\mathcal{O}\left(\tau\right)\right\} ,\]
\[
\theta_{1}\left(x\vert\frac{2i}{\tau}\right)=2\sin\pi x\exp\left(-\frac{\pi}{2\tau}\right)\left\{ 1+\mathcal{O}\left(\tau\right)\right\} \]
and\[
\Gamma\left(x\right)\Gamma\left(1-x\right)=\frac{\pi}{\sin\pi x},\quad x\notin\mathbb{Z}\]
to obtain\begin{align*}
\Gamma_{q}\left(x\right) & =\frac{\pi}{\sin\pi x}\frac{1}{\Gamma\left(1-x\right)}\left\{ 1+\mathcal{O}\left(\tau\right)\right\} =\Gamma\left(x\right)\left\{ 1+\mathcal{O}\left(\tau\right)\right\} \end{align*}
for $x<1$ and $x\notin\mathbb{Z}$ as $\tau\to0^{+}$. 
\end{proof}

\thanks{This work is partially supported by Chinese National Natural Science
Foundation grant No.10761002. The thanks for Tom Koornwinder for the
references \cite{Kirillov,Koornwinder1,Koornwinder2}.}

\end{document}